\documentclass[11pt]{amsart}
\usepackage{amsmath}
\usepackage{amsfonts}
\usepackage{amssymb}

\newtheorem{thm}{Theorem}

\newtheorem{cor}[thm]{Corollary}

\newtheorem{lemma}[thm]{Lemma}
\newtheorem{prop}[thm]{Proposition}

\theoremstyle{definition}
\newtheorem{defn}[thm]{Definition}

\theoremstyle{remark}
\newtheorem{remark}[thm]{Remark}

\newcommand{\bb}[1]{\mathbb{#1}}
\newcommand{\cl}[1]{\mathcal{#1}}

\usepackage[all]{xy}

\begin{document}

\title[Two Characterizations of the Maximal Tensor Product]{Two Characterizations of the Maximal Tensor Product of Operator Systems}

\author[Wai Hin ~Ng]{Wai Hin~Ng}
\address{Department of Mathematics, University of Houston,
Houston, Texas 77204-3476, U.S.A.}
\email{rickula@math.uh.edu}

\date{March 23, 2015}

\begin{abstract}
In this paper we provide two characterizations of the maximal tensor product structure for the category of operator systems introduced in \cite{KPTT1}. The first one is via the schur tensor product given in \cite{RKI}; the second one employs the idea of the CPAP in \cite{HP}. 
\end{abstract}

\maketitle


\section{Introduction}
An operator system is a self-adjoint and unital subspace of $\cl{B}(\cl{H})$ of all bounded operators on a Hilbert space $\cl{H}$. In recent years, the theory of tensor products of operator systems has been developed systematically, see e.g. \cite{Ka, KPTT1, KPTT2}. Given operator systems $\cl{S}$ and $\cl{T}$, their maximal tensor product, denoted by $\cl{S} \otimes_{\max} \cl{T}$, is equipped the smallest family of cones for which the algebra tensor product $\cl{S} \otimes \cl{T}$ forms an operator system. In the category of operator systems, the maximal tensor product is the natural analogue of the projective tensor norm of operator spaces, as well as a generalization of the maximal tensor norm of C*-algerbas. 

In fact, given any operator space $V$, there is an operator system $\cl{S}_V$ containing $V$ completely order isometrically, see \cite[Chp 8]{Pa} . In \cite{KPTT1}, it is shown that for operator spaces $V$ and $W$, the projective tensor product $V \overset{\wedge}{\otimes} W$ is completely isometrically included in $\cl{S}_V \otimes_{\max} \cl{S}_W$. Similary, any unital C*-algebras $\cl{A}$ and $\cl{B}$ are as well operator systems; in the same paper it is proved that their C*-maximal tensor product  $\cl{A} \otimes_{\text{C*-max}} \cl{B}$ is completely order isomorphic to $\cl{A} \otimes_{\max} \cl{B}$.  

Thus it is natural to characterize the maximal tensor product. In this paper we give two characterizations of the maximal tensor product from different approaches. The first approach given in section 3 is to examine the schur tensor product from \cite{RKI} in the category of operator systems. It provides a different view of the matricial cones of $\cl{S} \otimes_{\max} \cl{T}$. In section 4,  we employ factorization and the idea of the completely positive approximation property (CPAP) in \cite{HP} to characterize these matricial cones. We show that this characterization possesses connections to results on $(\min, \max)$-nuclearity found in \cite{Ka}. 

\section{Preliminaries}
We outline a few basic facts about the maximal tensor product and refer readers to \cite{FP, Ka, KPTT1, KPTT2} for the details. Given a pair of operator systems $(\cl{S}, \{P_n\}_{n=1}^{\infty}, 1_{\cl{S}} )$ and $(\cl{T}, \{Q_n \}_{n=1}^{\infty}, 1_{\cl{T}} )$, by an operator system structure on $\cl{S} \otimes \cl{T}$, we mean a family $\tau = \{ C_n \}_{n=1}^{\infty}$ of cones, where $C_n \subset M_n ( \cl{S} \otimes \cl{T} )$, satisfying:
	\begin{enumerate}
		\item[(T1)]		$(\cl{S} \otimes \cl{T}, \{ C_n \}_{n=1}^{\infty}, 1_{\cl{S}} \otimes 1_{\cl{T}} )$ is an operator system denoted by $\cl{S} \otimes_{\tau} \cl{T}$. 	
		\item[(T2)]		$P_n \otimes Q_m \subset C_{nm}$, for all $n, m \in \bb{N}$.
		\item[(T3)]		If $\phi \colon \cl{S} \to M_n$ and $\psi \colon \cl{T} \to M_m$ are unital completely positive maps, then $\phi \otimes \psi \colon \cl{S} \otimes_{\tau} \cl{T} \to M_{nm}$ is a unital completely positive map. 
	\end{enumerate}
By an operator system tensor product, we mean a mapping $\tau$ taking any pair of operator systems $\cl{S}$ and $\cl{T}$ into an operator system structure $\tau(\cl{S}, \cl{T})$, denoted by $\cl{S} \otimes_{\tau} \cl{T}$. We say $\tau$ is fuctorial, provided in addition it satisfies the following property:
	\begin{enumerate}
		\item[(T4)]		Given operator systems $\cl{S}_i$ and $\cl{T}_i$, $i = 1, 2$, if  $\phi_i \colon \cl{S}_i \to \cl{T}_i$ is unital completely positive, then $\phi_1 \otimes \phi_2 \colon \cl{S}_1 \otimes_{\tau} \cl{S}_2 \to \cl{T}_1 \otimes_{\tau} \cl{T}_2$ is unital completely positive. 
	\end{enumerate}
If for all operator systems $\cl{S}$ and $\cl{T}$, the map $\theta \colon x \otimes y \mapsto y \otimes x$ is a unital complete order isomorphism from $\cl{S} \otimes_{\tau} \cl{T}$ onto $\cl{T} \otimes_{\tau} \cl{S}$, then $\tau$ is a called symmetric.

We now recall the construction of the maximal tensor product. Given operator systems $\cl{S}$ and $\cl{T}$, we first define the family of cones
	\begin{align*}
		\cl{D}_n^{\max} ( \cl{S}, \cl{T} ) 		&= \{ A (P \otimes Q) A^* \colon P \in M_k( \cl{S} )^+, Q \in M_m( \cl{T} )^+, \\
														&	\qquad \qquad A \in M_{n, km}, k,m \in \bb{N} \}.
	\end{align*}
For short we denote $\cl{D}_n^{\max}(\cl{S}, \cl{T}) = \cl{D}_n^{\max}$. We shall remark the following useful representation of $\cl{D}_1^{\max}$.

\begin{lemma}
Every $u \in \cl{D}_1^{\max}$ can be represented as $u = \sum p_{ij} \otimes q_{ij}$ for some $[p_{ij}] \in M_n(\cl{S})^+$ and $[q_{ij}] \in M_n( \cl{T} )^+$.
\end{lemma}

\begin{proof}
If $u = A (P \otimes Q ) A^*$ as above with $A \in M_{1, km}$, note that $u$ is then the sum of the entries of the Kronecker tensor product $(A \left[\begin{smallmatrix} P & \dots & P \\ \vdots & & \vdots \\ P & \dots & P \end{smallmatrix}\right] A^*) \otimes Q$, where the operator matrix is in $M_{m}(M_k(\cl{S}))^+$. Since we can replace $Q$ by $\left[ \begin{smallmatrix} Q & 0 \\ 0 & 0 \end{smallmatrix} \right]$ of some appropriate size and likewise for the first operator matrix, we deduce such representation as claimed. 
\end{proof}

This matricial cone structure $\{\cl{D}_{n}^{\max}\}_{n=1}^{\infty}$ is then a compatible family with matrix order unit $1_{\cl{S}} \otimes 1_{\cl{T}}$. Yet it is not Archimedean, so we complete the cones through the Archimedeanization process (see \cite{PT}) basically by taking the closure of $\cl{D}_n^{\max}$: 
	\begin{equation*}
		\cl{C}_n^{\max} (\cl{S}, \cl{T} ) = \{ U \in M_n( \cl{S} \otimes \cl{T}) \colon \varepsilon (1_{\cl{S}} \otimes 1_{\cl{T}} ) + U \in \cl{D}_n(\cl{S}, \cl{T}), \forall r > 0 \}.
	\end{equation*}
Likewise we denote $\cl{C}_n^{\max}(\cl{S}, \cl{T}) = \cl{C}_n^{\max}$. Now $\cl{S} \otimes \cl{T}$ equipped with this family $\{\cl{C}^{\max}_n \}_{n=1}^{\infty}$ satisfies Properties (T1) to (T4) and it defines a symmetric and associative operator system structure. We call it the maximal tensor product of $\cl{S}$ and $\cl{T}$ and denote it $\cl{S} \otimes_{\max} \cl{T}$. 

The maximal tensor product is projective in the following sense. Let $\tau$ be an operator system tensor product. We say that $\tau$ is left projective, provided if $q \colon \cl{S} \to \cl{R}$ is a complete quotient map (\cite{FP, Ka}), then for any operator system $\cl{T}$, the map $q \otimes id$ is a complete quotient from $\cl{S} \otimes_{\tau} \cl{T}$ onto $\cl{R} \otimes_{\tau} \cl{T}$. It is equivalent to require that for every $n \in \bb{N}$, every $u \in M_n( \cl{R} \otimes_{\tau} \cl{T} )^+$, and every $\varepsilon > 0$, there is $\tilde{u_{\varepsilon}} \in M_n( \cl{S} \otimes_{\tau} \cl{T} )^+$ so that $q \otimes id (\tilde{u_{\varepsilon}}) = u + \varepsilon (I_n \otimes 1_{\cl{R}} \otimes 1_{\cl{T}})$. Right projectivity is defined similarly and we say $\tau$ is projective if it is both left and right projective. 

This maximal tensor product has the following universal property:

\begin{thm}\cite[Theorem 5.8]{KPTT1}
Let $\cl{S}$ and $\cl{T}$ be operator systems. A bilinear map $\phi \colon \cl{S} \times \cl{T} \to \cl{B}(\cl{H} )$ is jointly completely positive if and only if its linearization $L_{\phi} \colon \cl{S} \otimes_{\max} \cl{T} \to \cl{B}(\cl{H})$ is a completely positive map. Moreover, if $\tau$ is an operator system structure on $\cl{S} \otimes \cl{T}$ satisfying this property, then $\cl{S} \otimes_{\tau} \cl{T} = \cl{S} \otimes_{\max} \cl{T}$. 
\end{thm}

If we take $\cl{B}(\cl{H}) = \bb{C}$, we obtain the following representation of the maximal tensor product:
	\begin{equation*}
		( \cl{S} \otimes_{\max} \cl{T} )^{d, +}	= CP( \cl{S}, \cl{T}^d ),
	\end{equation*}
where the latter set is the cone of all completely positive maps from $\cl{S}$ to $\cl{T}^d$. This statement is precisely the operator system analogue of a result by Lance in \cite{La}.

The following lemma is in \cite{KPTT1} and will be used in the next section. We include the proof for completion.

\begin{lemma}
Let $\cl{S}$ and $\cl{T}$ be operator systems and $\{C_n \}_{n=1}^{\infty}$ be a compatible family of cones of $\cl{S} \otimes \cl{T}$ satisfying Property (T2). Then $\cl{D}_n^{\max} \subset C_n$. 
\end{lemma}

\begin{proof}
If $P \in M_n( \cl{S} )^+$ and $Q \in M_m(\cl{T})^+$, then Property (T2) implies $P \otimes Q \in C_{nm}$. By compatibility of $\{C_n\}_{n=1}^{\infty}$, $A(P \otimes Q)A^* \in C_{k}$, for all $A \in M_{k, nm}$; hence $\cl{D}_n^{\max} \subset C_n$. 
\end{proof}


\section{The Schur Tensor Product}
In this section we examine the schur tensor product from \cite{RKI} in the category of operator systems. It turns out that in the operator system settings, the matricial cones of the schur tensor product coincide with that of the maximal tensor product, providing a different description of the maximal tensor product.  

\begin{defn}
Given operator systems $\cl{S}$ and $\cl{T}$, $X = [ x_{ij} ] \in M_n( \cl{S} )^+$, and $Y = [ y_{ij}] \in M_n( \cl{T} )^+$, we define the schur tensor product $X \circ Y$ to be 
	\begin{equation*}
			X \circ Y 	:=	[ x_{ij} \otimes y_{ij} ] \in M_n( \cl{S} \otimes \cl{T}). 
	\end{equation*}
\end{defn}

\begin{lemma}
Every $X \circ Y \in M_n( \cl{S} \otimes \cl{T} )$ can be regarded as $A( X \otimes Y ) A^*$, for some $A \in M_{n, n^2}$. 
\end{lemma}

\begin{proof}
Let $\{E_{ij}\}_{i,j=1}^n$ denote the standard matrix units of $M_{n}(\bb{C})$ and regard $X \otimes Y$ as the Kronecker tensor product. In the case when $n = 2$, note that $\begin{bmatrix} E_{11} & E_{22} \end{bmatrix}  X \otimes Y  \begin{bmatrix} E_{11} & E_{22} \end{bmatrix}^*=$
	\begin{align*}
		&		\begin{bmatrix}
					E_{11}		& 	E_{22} 
				\end{bmatrix}
				\begin{bmatrix}
					x_{11} \otimes y_{11} 	& x_{11} \otimes y_{12} & x_{12} \otimes y_{11} & x_{12} \otimes y_{12} \\
					x_{11} \otimes y_{21}	&	x_{11} \otimes y_{22}	&	x_{12} \otimes y_{21} 	&	x_{12} \otimes y_{22}	\\
					x_{21} \otimes y_{11}	&	x_{21} \otimes y_{12}	&	x_{22} \otimes y_{11} 	&	x_{22} \otimes y_{12}	\\					
					x_{21} \otimes y_{21}	&	x_{21} \otimes y_{22}	&	x_{22} \otimes y_{21} 	&	x_{22} \otimes y_{22}	\\
				\end{bmatrix}
				\begin{bmatrix}
					E_{11} 	\\	E_{22}
				\end{bmatrix}
				\\
		&\qquad \qquad \qquad	\qquad =	\begin{bmatrix}
						x_{11} \otimes y_{11}	&		x_{12} \otimes y_{12}\\
						x_{21} \otimes y_{21}	&		x_{22} \otimes y_{22}
				\end{bmatrix}
				=		X \circ Y.
	\end{align*}
In general, we may view $X \circ Y$ as a pre-and-post mulitplication of $X \otimes Y$ by a special $n \times n^2$ matrix $\cl{E} = \begin{bmatrix} E_{11} & E_{22} & \dots & E_{nn}		\end{bmatrix}$, a similar calculation shows that  $X \circ Y = \cl{E} (X \otimes Y) \cl{E}^*$. 
 \end{proof}

\begin{lemma}
Every $P \in M_n ( \cl{S} \otimes \cl{T} )$ can be written as $P = A ( X \circ  Y ) B$, for some $X \in M_k( \cl{S} )$, $Y \in M_k( \cl{T} )$, $A \in M_{n, k}$ and $B \in M_{k, n}$. In particular, we may take $B = A^*$. 
\end{lemma}

\begin{proof}
Write $P$ as a sum of matrices whose entries are elementary tensors, that is, $P = \sum_{ l = 1}^m U^l $, where $U^l = [ x_{ij}^l \otimes y_{ij}^l ] \in M_n( \cl{S} \otimes \cl{T} )$. Let $U = U^1 \oplus \dots \oplus U^m$, so 
	\begin{equation*}
			U		=		\begin{bmatrix}
								U^1		&		0		&	\dots &		0		\\
								0			&		U^2	&	\dots 	&		0	\\
								\vdots	&		\vdots & \ddots &	\vdots	\\
								0			&		0		&		\dots 	&	U^m
							\end{bmatrix}
	\end{equation*}
which is $X \circ Y$ for some $X \in M_{nm}( \cl{S} )$ and $Y \in M_{nm}( \cl{T} )$. Now let $A = \begin{bmatrix}	I_n  & I_n & \dots & I_n		\end{bmatrix} \in M_{n, nm}$ with $m$ copies of $I_n$. Then,
	\begin{align*}
		(A U) A^*		&=	\begin{bmatrix}		U^1 & U^2 & \dots & U^l	\end{bmatrix}_{n \times nm} \begin{bmatrix}	 I_n \\ I_n \\ \vdots \\ I_n		\end{bmatrix}_{nm \times n}		=	\sum_{l=1}^m U^l	= P. 
	\end{align*}
\end{proof}

Lemma 5 shows that schur tensor product is in fact of the form of elements in $\cl{D}_n^{\max}$, except positivity. Motivated by Lemma 6 and the construction of the maximal tensor product, we define the following family of cones. 

\begin{defn}\label{schurCone}
Given operator systems $\cl{S}$ and $\cl{T}$, we define
	\begin{align*}
		\cl{C}_n^s (\cl{S} \otimes \cl{T} ) 	&:=	\{	A (X \circ Y ) A^* 	\in M_n( \cl{S} \otimes \cl{T} )	\colon 			\nonumber	\\		
																	&\qquad \qquad X \in M_k(\cl{S})^+, Y \in M_k(\cl{T})^+, A \in M_{n, k}	, k \in \bb{N}		\}.
	\end{align*}
For short we denote $\cl{C}_n^s (\cl{S} \otimes \cl{T}) = \cl{C}_n^s$. 
\end{defn}

\begin{prop}
The family $\{ \cl{C}_n^s \}_{n=1}^{\infty}$ defines a matrix ordering on $\cl{S} \otimes \cl{T}$ with matrix order unit $1 \otimes 1$. 
\end{prop}

\begin{proof}
We first check that $\cl{C}_n^s$ is a cone of $M_n( \cl{S} \otimes \cl{T} )$. It is obvious from definition that $\cl{C}_n^s \subset M_n( \cl{S} \otimes \cl{T} )_{sa}$. Let $A ( X_1 \circ Y_1 ) A^*$ and $B ( X_2 \circ Y_2 ) B^*$ be in $\cl{C}_n^s$, where $X_1 \in M_k( \cl{S} )$,  $Y_1 \in M_k( \cl{T} )$, $X_2 \in M_m( \cl{S})$, $Y_2 \in M_m( \cl{T})$, $A \in M_{n, k}(\bb{C})$, and $B \in M_{n, m}(\bb{C})$. Let
	\begin{equation*}
			X 		=	 X_1 \oplus X_2 = 	\begin{bmatrix}		X_1 	&	0	\\	0	&	X_2		\end{bmatrix} \in M_{k+m}( \cl{S} )^+,
	\end{equation*}
and likewise	$Y = Y_1 \oplus Y_2 \in M_{k +m} (\cl{T})^+$. Consider $\begin{bmatrix}	A & B \end{bmatrix} \in M_{n, k+m}$, then 
	\begin{align*}
			\begin{bmatrix}	A & B \end{bmatrix}  ( X  \circ Y ) \begin{bmatrix}	A & B \end{bmatrix} ^*	
					&=	\begin{bmatrix}	A & B \end{bmatrix}  \begin{bmatrix}
									X_1 \circ Y_1		&		0		\\
									0		&		X_2 \circ Y_2
							\end{bmatrix}
							\begin{bmatrix}	A & B \end{bmatrix}^*	\\
					&=	A (X_1 \circ Y_1) A^* + B( X_2 \circ Y_2 ) B^*
	\end{align*}
is in $\cl{C}_n^s$. If $t > 0$, then $t ( A ( X \circ Y ) A^* ) = ( \sqrt{t} A) ( X \circ Y ) ( \sqrt{t} A)^* \in \cl{C}_n^s$. Also, if $B \in M_{r, n}$ then $(BA)( X \circ Y ) (B A)^* \in \cl{C}_r^s$. Therefore, $\{ \cl{C}_n^s \}_{n=1}^{\infty}$ is a compatible family of cones on $\cl{S} \otimes \cl{T} )$.

Finally, to see that they are proper, we claim that in fact $\cl{C}_n^{s} \subset \cl{D}_n^{\max}$. Indeed, let $A( X \circ Y ) A^* \in \cl{C}_n^s$, for some $X \in M_k(\cl{S})^+$, $Y \in M_k( \cl{T} )^+$, and $A \in M_{n, k}$. Then by the previous lemma,
	\begin{align*}
		A (X \circ Y ) A^* 		&=		A( \cl{E} ( X \otimes Y) \cl{E}^* ) A^* 	=	( A \cl{E} )( X \otimes Y ) ( A \cl{E} )^*,
	\end{align*}
which is in $\cl{D}_n^{\max}(\cl{S}, \cl{T})$ by definition. Since the latter cone is proper, $- \cl{C}_n^s \cap \cl{C}_n^s = \{ 0 \}$. The fact that $1 \otimes 1$ is a matrix order unit with respect to $\{ \cl{C}_n^s \}$ follows from the inclusion $\cl{C}_n^s \subset \cl{D}_n^{\max}$ and that $1 \otimes 1$ is a matrix order unit with respect to $\cl{D}_{n}^{\max}$. Consequently, $\{ \cl{C}_n^s \}_{n=1}^{\infty}$ defines a matrix ordering on $\cl{S} \otimes \cl{T}$.
\end{proof}

From the last paragraph of the proof, we see that $\cl{C}_n^s \subset \cl{D}_n^{\max}$. In fact, one can further deduce that $\cl{C}_n^s  = \cl{D}_n^{\max}$ by Lemma 3, after proving that this family satisfies Property (T2). 

\begin{lemma}
The family $\{ \cl{C}_n^s \}_{n=1}^{\infty}$ satisfies Property (T2). That is, given $X \in M_n( \cl{S} )^+$ and $Y \in M_m( \cl{T} )^+$, $X \otimes Y \in \cl{C}_{nm}^s$. 
\end{lemma}

\begin{proof}
Let $X$ and $Y$ be as above, note that we may view 
	\begin{align*}
		X \otimes Y 	&=	[  x_{ij}  \otimes Y ]_{i,j=1}^n		\\
								&=	\begin{bmatrix}	x_{11} \otimes J_m		&	\dots & 	x_{n1} \otimes J_m	\\	\vdots	& 	\ddots & \vdots	\\ x_{n1} \otimes J_m & \dots & x_{nn} \otimes J_m		\end{bmatrix}_{nm \times nm}	\circ 	\begin{bmatrix}	Y & \dots & Y	\\ \vdots & \ddots & \vdots \\ Y 	& \dots & Y	\end{bmatrix}_{nm \times nm},
	\end{align*}
where $J_k \in M_k( \bb{C} )$ is the matrix of entries all 1. It is easy to see that the second matrix in the above equation is $Y \otimes J_n$. A straight-forward calculation shows that for each $k \in \bb{N}$, $J_k$ has eigenvalues $0$ and $k$, so $Y \otimes J_n \in M_{nm}( \cl{T} )^+$. On the other hand, after the ``canonical shuffle'' \cite[Chp 3]{Pa}, the first matrix is unitarily equivalent to $X \otimes J_m$, which is also positive in $M_{nm} (\cl{S})$. Therefore, $X \otimes Y = ( X \otimes J_m ) \circ ( Y \otimes J_n ) \in \cl{C}_{nm}^s$ and the family $\{ \cl{C}_n^s \}_{n=1}^{\infty}$ satisfies Property (T2).
\end{proof}

\begin{remark}
Now by Lemma 3 we have the reverse inclusion $\cl{D}_n^{\max} \subset \cl{C}_n^s $, so the two families of cones are the same. In particular, Lemma 1 follows easily: every $u \in \cl{D}_1^{\max} = \cl{C}_1^s$ can be represented as $u = A (P \circ Q) A^* = (A^* P A) \circ Q$, for some $A \in M_{1, n}$, $P \in M_n(\cl{S})^+$, and $Q \in M_n(\cl{T})^+$.  If we archimedeanize the cones $\{ \cl{C}_n^s \}_{n=1}^{\infty}$, then we obtain the schur tensor product of operator systems and denote it $\cl{S} \otimes_s \cl{T}$; and it is unitally completely order isomorhpic to $\cl{S} \otimes_{\max} \cl{T}$. 
\end{remark}

\begin{thm}
The cones $\cl{C}_n^{s} = \cl{D}_n^{\max}$, for every $n \in \bb{N}$. Consequently, for operator systems, the schur tensor product is the maximal tensor product, i.e. $\cl{S} \otimes_s \cl{T} = \cl{S} \otimes_{\max} \cl{T}$.
\end{thm}

Given operator systems $\cl{S}$ and $\cl{T}$,  $\cl{S} \otimes_{\max} \cl{T}$ when viewed as an operator space, possesses a natural operator space matrix norm $|| \cdot ||_{\text{osy-max}}$; that is, given $U \in M_n( \cl{S} \otimes_{\max} \cl{T} )$, 
	\begin{equation*}
		|| U ||_{\text{osy-max}} 	= 	\inf \left\{	r \colon \begin{bmatrix}	 r I & U \\ U^* & r I	\end{bmatrix} \in M_{2n} ( \cl{S} \otimes_{\max} \cl{T} )^+	 \right\}.
	\end{equation*}

In particular since $\cl{A} \otimes_{\text{C*-max}} \cl{B} = \cl{A} \otimes_{\max} \cl{B}$ for C*-algebras, the C*- maximal tensor norm $|| \cdot ||_{\text{C*-max}}$ is precisely $|| \cdot ||_{\text{osy-max}}$ for C*-algebras. The following proposition is a slightly generalized version of $|| \cdot ||_{\text{C*-max}} \leq || \cdot ||_s$ in \cite{RKI}. 

\begin{prop}\label{schurcontraction}
Let $\cl{S}$ and $\cl{T}$ be operator systems. Then the identity map $\phi \colon \cl{S} \otimes^s \cl{T} \to \cl{S} \otimes_{\max} \cl{T}$ is a complete contraction.
\end{prop}

\begin{proof}
Let $|| U ||_s < 1$, then by scaling, there exist scalar contractions $A, B$ and $X \in M_n(S)$ and $Y\in M_n(T)$, $|| X ||, || Y || \leq 1$ such that $U = A(X \circ Y)B$.
Hence, the matrices $P = \begin{bmatrix} I & X\\X^* & I \end{bmatrix} \in M_{2n}(S)^+$ and $Q = \begin{bmatrix} I & Y\\Y^* & I \end{bmatrix} \in M_{2n}(T)^+$. Note that
	\begin{align*}
		\begin{bmatrix} A & 0 \\ 0 & B^* \end{bmatrix} P \circ Q \begin{bmatrix} A^* & 0 \\ 0 & B \end{bmatrix} &= \begin{bmatrix} AA^* & A( X \circ Y) B \\ B^* ( X^* \circ Y^* ) A^* & B^*B	\end{bmatrix}	\\
			&=	\begin{bmatrix}	AA^* & U \\ U^* & B^*B	\end{bmatrix},
	\end{align*}
which is in $M_{2n} (\cl{S} \otimes_{s} \cl{T} )^+ = M_{2n}( \cl{S} \otimes_{\max} \cl{T})^+$. 

On the other hand, since $A$ and $B$ are scalar contractions, $I - A A^*$ and $I - B^* B$ are positive in $M_n$. Thus, the operator matrix $\begin{bmatrix}	I - AA^* & 0 \\ 0 & I - B^*B 	\end{bmatrix}$ is positive in $M_{2n}( \cl{S} \otimes_{\max} \cl{T} )$. By adding the two matrices, we obtain  $\begin{bmatrix} I & U\\U^* & I \end{bmatrix} \in M_{2n}(\cl{S} \otimes_{\max} \cl{T})^+$ which implies that $||U||_{\text{osy-max}} \leq 1.$
\end{proof}


\section{Factorization Through The Matrix Algebras $M_n$}
We now turn to study the maximal tensor product using factorization. Recall that every $u = \sum_{i=1}^n x_i \otimes y_i \in \cl{S} \otimes \cl{T}$ maybe regarded as a linear map $\hat{u} \colon \cl{S}^d \to \cl{T}$, $\hat{u} (f) = \sum_{i=1}^n f(x_i)y_i$, where $\cl{S}^d$ is the linear dual of $\cl{S}$. The map $\hat{u}$ is independent of representation of $u$ and $u \mapsto \hat{u}$ is a one-to-one correspondence between $\cl{S} \otimes \cl{T}$ and $L(\cl{S}^d, \cl{T})$, where the latter is the space of linear maps from the linear dual $\cl{S}^d$ to $\cl{T}$. 

In this section, we use the duality results from \cite{FP}. Henceforth, to ensure $\cl{S}^d$ is an operator system, we assume $\cl{S}$ and $\cl{T}$ to be finite dimensional. Fix a basis $\{ y_1 = 1_{\cl{T}}, \dots, y_m \}$ for $\cl{T}$, where $y_i = y_i^*$ and $|| y_i || = 1$, so that every $ u \in \cl{S} \otimes \cl{T}$ has a unique representation $u= \sum_{i=1}^m x_i \otimes y_i$, for some $x_i \in \cl{S}$. To obtain the main result in this section, we introduce a temporary norm on $\cl{S} \otimes \cl{T}$ by setting $||| u ||| = \sum_{i=1}^m || x_i ||$.

\begin{lemma}
If $u = \sum_{i=1}^m x_i \otimes y_i \in \cl{S} \otimes \cl{T}$, where $x_i = x_i^*$, then $||| u ||| (1_{\cl{S}} \otimes 1_{\cl{T}}) + u \in \cl{D}_1^{\max}( \cl{S}, \cl{T} )$. 
\end{lemma}

\begin{proof}
Because
	\begin{equation*}
		\begin{bmatrix}
			|| s_i || 1	&		s_i	\\ s_i 	& || s_i || 1
		\end{bmatrix} \in M_2( \cl{S} )^+,	\qquad 
		\begin{bmatrix}
				1		&		t_i 	\\	t_i 	&		1
		\end{bmatrix}	\in M_2( \cl{T} )^+,
	\end{equation*}	
when we form their schur tensor product, we obtain
	\begin{equation*}
		\begin{bmatrix}
			|| s_i || 1 \otimes 1	& 	s_i \otimes t_i 	\\	s_i \otimes t_i & || s_i || 1 \otimes 1
		\end{bmatrix}	\in \cl{D}_2^{\max}( \cl{S}, \cl{T} )^+.
	\end{equation*}
Pre-and-post multiply this matrix by $[1, 1]$ shows that  $|| s_i || (1 \otimes 1) + s_i \otimes t_i \in \cl{D}_1(\cl{S}, \cl{T} )$ for each $i$, thus the sum $||| u ||| ( 1 \otimes 1) + u \in \cl{D}_1^{\max}( \cl{S}, \cl{T} )$. 
\end{proof}

\begin{lemma}
Let $u_{\lambda}$ be a net in $\cl{S} \otimes \cl{T}$. Then $||| u_{\lambda} ||| \to 0$ in $\cl{S} \otimes \cl{T}$ if and only if for each $f \in \cl{S}^d$, $|| \hat{u_{\lambda}} (f) ||_{\cl{T}} \to 0$.
\end{lemma}

\begin{proof}
Since every $u_{\lambda}$ has a unique representation $u_{\lambda} = \sum_{i=1}^m x_i^{\lambda} \otimes y_i$, $||| u_{\lambda} ||| \to 0$ implies that $\lim_{\lambda} || x_i^{\lambda} || \to 0$ for each $i \in \{ 1, \dots m \}$, which is equivalent to require that $x_i^{\lambda} \to 0$ in the weak topology. Thus for each $f \in \cl{S}^d$,
	\begin{align*}
		|| \hat{u_{\lambda}} (f) ||_{\cl{T}} \leq \sum_{i=1}^m | f (x_i^{\lambda} ) | \cdot || y_i ||_{\cl{T}}	\to 0.
	\end{align*}

Conversely, it suffices to show that for each $i \in \{1, \dots, m\}$, $\lim_{\lambda} || x_i^{\lambda} || = 0$. Note that for $t = \sum_{i=1}^m c_i y_i \in \cl{T}$, $\alpha(t) := \sum_{i=1}^m | c_i |$ defines a norm on $\cl{T}$. Since $\cl{T}$ is finite dimensional, $|| t ||_{\cl{T}} \leq \alpha(t) \leq K || t ||_{\cl{T}}$ for some $K > 0$. For each $f \in \cl{S}^d$, taking $c_i = f(x_i^{\lambda})$ shows that
	\begin{equation*}
		\sum_{i=1}^m | f(x_i^{\lambda}) | \leq K || \hat{u}_{\lambda} (f) ||_{\cl{T}} \to 0.
	\end{equation*}
Hence for each $f \in \cl{S}^d$ and $i \in \{1, \dots, n \}$, $| f(x_i^{\lambda}) | \to 0$. The latter condition is equivalent to $(x_i^{\lambda}) \to 0$ in the weak topology, which coincides with the norm topology because $\cl{S}$ is finite dimensional. 
\end{proof}

\begin{defn}
A linear map $\theta \colon \cl{S} \to \cl{T}$ factors through $M_n$ approximately, provided there exists nets of completely positive maps $\phi_{\lambda} \colon \cl{S} \to M_{n_{\lambda}}$ and $\psi_{\lambda} \colon M_{n_{\lambda}} \to \cl{T}$ such that $\psi_{\lambda} \circ \phi_{\lambda}$ converges to $\theta$ in the point-norm topology. An operator system $\cl{S}$ is said to have complete positive approximation property (CPAP) if the identity map factors through $M_n$ approximately.
\end{defn}

In \cite{HP} it is shown that $\cl{S}$ is $(\min, \max)$-nuclear if and only if $\cl{S}$ has CPFP. We now establish the main theorem in the section. 

\begin{thm}\label{CharMax}
Let $\cl{S}$ and $\cl{T}$ be finite dimensional operator systems and $u \in (\cl{S} \otimes_{\max} \cl{T})^+$. The following are equivalent:
	\begin{enumerate}
		\item 	$u$ is positive in $\cl{S} \otimes_{\max} \cl{T}$.
		\item 	The map $\hat{u} \colon \cl{S}^d \to \cl{T}$ factors through $M_n$ approximately:
			\begin{equation*}
					\xymatrix{
						\cl{S}^d 	\ar[rr]^{\hat{u}}	\ar[rd]_{\varphi_{\lambda}}	&			&		\cl{T}		\\
										&		M_{n_{\lambda}}		\ar[ru]_{\psi_{\lambda}}		&
						}
			\end{equation*}

	\end{enumerate}
\end{thm}

\begin{proof}
Suppose $u \in (\cl{S} \otimes_{\max} \cl{T})^+$. Then for each $\varepsilon > 0$, $u_{\varepsilon} = \varepsilon (1 \otimes 1) + u$ is in $\cl{D}_1^{max}$.By Lemma 1 it can be written as $u_{\varepsilon} = \sum p_{ij}^{\varepsilon} \otimes q_{ij}^{\varepsilon}$, where $P_{\varepsilon} = [p_{ij}^{\varepsilon} ] \in M_{n_{\varepsilon}} ( \cl{S} )^+$ and $Q_{\varepsilon} = [q_{ij}^{\varepsilon}] \in M_{n_{\varepsilon}} ( \cl{T} )^+$. Define $\varphi_{\varepsilon} \colon \cl{S}^d \to M_{n_{\varepsilon}}$ by $\varphi_{\varepsilon}(f) = [ f ( p_{ij}^{\varepsilon} ) ]$ and $\psi_{\varepsilon} \colon M_{n_{\varepsilon}} \to \cl{T}$ by $\psi_{\varepsilon} ( [a_{ij}] ) = \sum a_{ij} q_{ij}^{\varepsilon}$. 
Note that $\varphi_{\varepsilon}$ is completely positive by definition of $\cl{S}^d$. For $\psi_{\varepsilon}$, first consider  the completely positive map $[a_{ij}] \mapsto [a_{ij} ] \otimes Q_{\varepsilon}$. We regard $[a_{ij} ] \otimes Q_{\varepsilon}$ as the matrix $[ q_{ij}^{\varepsilon} [ a_{kl} ] ]_{i,j}^{n_{\varepsilon}}$ and pre-and-post multiply it by $[E_{11}, E_{12}, \dots, E_{1 n_{\varepsilon} } ]$, then we obtain the matrix $[ q_{ij}^{\varepsilon} a_{ij} ]_{i, j=1}^{n_\varepsilon} \in M_{n_{\varepsilon}}( \cl{T} )^+$. Now pre-and-post multiply it by the row vector of length $n_{\varepsilon}$ whose entries are $1$; this yields $\sum_{i,j} a_{ij} q_{ij}^{\varepsilon}$, and $\psi_{\varepsilon}$ is completely positive.  It follows that $\hat{u_{\varepsilon}} = \psi_{\varepsilon} \circ \varphi_{\varepsilon}$ and it converges to $\hat{u}$ as $\varepsilon \to 0$ in the point-norm topology.

Conversely,  every $\psi_{\lambda} \circ \varphi_{\lambda}$ corresponds to a $w_{\lambda} \in \cl{S} \otimes \cl{T}$ so that $\hat{w_{\lambda}} = \psi_{\lambda} \circ \varphi_{\lambda}$. Identifying $\varphi_{\lambda} $ to $P_{\lambda} = [p_{ij}^{\lambda}] \in M_{n_{\lambda}}( \cl{S} )^+$ and $\psi_{\lambda}$ to $Q_{\lambda} = [q_{ij}^{\lambda}] \in M_{n_{\lambda}}( \cl{T})^+$ shows that $w_{\lambda} = \sum_{i,j}^{n_{\lambda}} p_{ij}^{\lambda} \otimes q_{ij}^{\lambda} \in \cl{D}_1^{\max}( \cl{S} \otimes \cl{T} )$. By the point-norm convergence and the last lemma, $\lim_{\lambda} ||| u - w_{\lambda} ||| \to 0$. Now for each $\lambda$, take $\varepsilon_{\lambda} = ||| u - w_{\lambda} |||$, and Lemma 13 asserts that $ \varepsilon_{\lambda} (1 \otimes 1) + (u - w_{\lambda}) \in \cl{D}_1^{\max}( \cl{S}, \cl{T} )$. For each $\varepsilon > 0$ there exists a $\lambda$, so that $\varepsilon_{\lambda} < \varepsilon$ so $ \varepsilon (1 \otimes 1) + (u - w_{\lambda}) \in \cl{D}_1^{\max}( \cl{S}, \cl{T} )$.  Hence $\varepsilon 1 \otimes 1 + u \in \cl{D}_1^{\max} (\cl{S}, \cl{T} )$ and $u \in (\cl{S} \otimes_{\max} \cl{T})^+$. 
\end{proof}

By the identification $M_m( \cl{S} \otimes_{\max} \cl{T} ) \cong \cl{S} \otimes_{\max} M_m( \cl{T} )$, we establish the following characterization of the matricial cone structure of the maximal tensor product. 

\begin{thm}
An element $U \in  M_m( \cl{S} \otimes_{\max} \cl{T} )$ is positive if and only if $\hat{U} \colon \cl{S}^d \to M_m(\cl{T})$ factors through $M_n$ approximately. 
\end{thm}

We would like to remark that this result is rather interesting. In \cite{FP} we have $\cl{S} \otimes_{\min} \cl{T} = (\cl{S}^d \otimes_{\max} \cl{T}^d )^d$. Combining with the result after Theorem 1, we deduce that $(\cl{S} \otimes_{\min} \cl{T})^+ = CP(\cl{S}^d, \cl{T})$; whereas by Theorem \ref{CharMax}, $(\cl{S} \otimes_{\max} \cl{T})^+$ corresponds to a proper subcone of $CP(\cl{S}^d, \cl{T})$ whose elements factor through $M_n$ approximately.  Since the minimal and maximal tensor products each represents respectively the largest and smallest matricial cone structure one can equip on $\cl{S} \otimes \cl{T}$, it brings up the natural question about the corresponding subsets with respect to other tensor products in \cite{KPTT1}.

Symmetry and projectivity of the maximal tensor product can also be obtained by this diagram. 

\begin{prop}
The maximal tensor product is symmetric and projective.
\end{prop}

\begin{proof}
Let $u = \sum s_i \otimes t_i \in \cl{S} \otimes_{\max} \cl{T}$. By dualizing the diagram in Theorem \ref{CharMax}, one sees that 
	\begin{equation*}
		\xymatrix{
			\cl{T}^d		\ar[dr]_{\psi_{\lambda}^d} \ar[rr]^{ (\hat{u})^d}	&			&		\cl{S}^{dd} = \cl{S}	\\
						&		M_{n_{\lambda}}^d = M_{n_{\lambda}} \ar[ru]_{\varphi_{\lambda}^d}			&
		}
	\end{equation*}
where $(\hat{u})^d$ is the map $g \mapsto \sum g(t_i) s_i$. Consequently, $\sum t_i \otimes s_i \in (\cl{T} \otimes_{\max} \cl{S})^+$ if and only if the above diagram holds, which by duality is equivalent to Theorem \ref{CharMax} (2). This shows that $\cl{S} \otimes_{\max} \cl{T} \cong_{\max} \cl{T} \otimes \cl{S}$ at the ground level. At each matrix level $n$, identifying $M_n(\cl{S} \otimes_{\max} \cl{T} ) = \cl{S} \otimes_{\max} M_n(\cl{T} )$ and replacing $\cl{T}$ by $M_n(\cl{T})$ proves symmetry of the maximal tensor product. 

For projectivity, first consider a complete quotient map $q \colon \cl{S} \to \cl{R}$. We claim that every $u \in (\cl{R} \otimes_{\max} \cl{T} )^+$ can be lifted to some $w \in ( \cl{S} \otimes_{\max} \cl{T} )^+$. Indeed, by Theorem \ref{CharMax} there are $\varphi_{\lambda}$ and $\psi_{\lambda}$ such that $ \psi_{\lambda} \circ \varphi_{\lambda} $ converges to $u$ in the point-norm topology. Since $q^d \colon \cl{R}^d \to \cl{S}^d$ is a complete order inclusion, by the Arveson extension theorem, there is a completely positive $\Phi_{\lambda} \colon \cl{S}^d \to M_{n_{\lambda}}$ extending $\varphi_{\lambda}$. Hence, the following diagram commutes:
			\begin{equation*}
					\xymatrix{
						\cl{R}^d 	\ar[rrr]^{\hat{u}}	\ar[d]_{q^d}	\ar[rrd]_{\varphi_{\lambda}}	&		&			&		\cl{T}		\\
						\cl{S}^d	\ar@{-->}[rr]_{\Phi_{\lambda}}		&			& 	M_{n_{\lambda}}		\ar[ru]_{\psi_{\lambda}}		&
						}
			\end{equation*}
Let $[ s_{ij}^{\lambda} ] \in M_{n_{\lambda}} ( \cl{S})^+$ be the corresponding matrix of $\Phi_{\lambda}$ and likewise for $[t_{ij}^{\lambda} ] \in M_{n_{\lambda}}( \cl{T} )^+$ of $\psi_{\lambda}$. Then $w_{\lambda} = \sum_{i,j} s_{ij}^{\lambda} \otimes t_{ij}^{\lambda} \in ( \cl{S} \otimes_{\max} \cl{T} )^+$ by the schur characterization and $\hat{w} = \psi_{\lambda} \circ \Phi_{\lambda}$. To this end, we claim that there is a subnet $w_{\lambda_{\alpha}}$ converging to some positive $w$ such that $\hat{w} \circ q^d = \hat{u}$. 

Let $\delta_0$ denote the unit in $\cl{R}^d \subset_{coi} \cl{S}^d$. Then $|| \hat{w_{\lambda}} (\delta_0) || = || \psi_{\lambda} \circ \varphi_{\lambda} (\delta_0) || \to || \hat{u} (\delta_0 ) ||$ asserts there is $\lambda_0$ such that the set $\{ || \hat{w_{\lambda}} ( \delta_0 )  || \colon \lambda > \lambda_0 \}$ is bounded. However, for completely positive maps, $|| \hat{w_{\lambda}} (\delta_0) || = || \hat{w_{\lambda}} ||_{cb} = || \hat{w_{\lambda} } ||$ and the latter norm also defines a norm on $\cl{S} \otimes \cl{T}$. By the equivalence of norm topologies, $\{ w_{\lambda} \colon \lambda > \lambda_0 \}$ is bounded in $(\cl{S} \otimes_{\max} \cl{T})^+$ and possesses a convergent subnet $w_{\lambda_{\alpha}} \to w \in (\cl{S} \otimes_{\max} \cl{T})^+$. Therefore for each $f \in \cl{R}^d$, 
	\begin{align*}
		|| ( \hat{w} \circ q^d - \hat{u} ) f || &= || ( \lim_{\alpha} \hat{w_{\lambda_{\alpha}} } \circ q^d - \hat{u} )f || = || (\lim_{\alpha} \psi_{\lambda_{\alpha}} \circ  \Phi_{\lambda_{\alpha} } \circ q^d - \hat{u} ) 	f||\\
									&=	|| ( \lim_{\alpha} \psi_{\lambda_{\alpha}} \circ \varphi_{\lambda_{\alpha}} - \hat{u}  ) f ||= \lim_{\alpha} || (\psi_{\lambda_{\alpha}} \circ \varphi_{\lambda_{\alpha}} - \hat{u} ) f ||  \\
									&=	\lim_{\lambda} || ( \psi_{\lambda} \circ \varphi_{\lambda} - \hat{u} ) f || \to 0,
	\end{align*}
where the second line follows from Lemma 14. Consequently every positive $u \in \cl{R} \otimes_{\max} \cl{T}$ can be lifted to a positive $w \in \cl{S} \otimes_{\max} \cl{T}$. This implies that for every such $u$ and for each $\varepsilon > 0$,  the element $w + \varepsilon (1_{\cl{S}} \otimes 1_{\cl{T}} ) \in ( \cl{S} \otimes_{\max} \cl{T} )^+$ satisfies $(q \otimes id)(w + \varepsilon(1_{\cl{S}} \otimes 1_{\cl{T}}) ) = u + \varepsilon (1_{\cl{R}} \otimes 1_\cl{T} )$. 

Finally, again by identifying $M_n( \cl{R} \otimes_{\max} \cl{T})$  to $\cl{R} \otimes_{\max} M_n(\cl{T})$ and likewise for $\cl{S} \otimes_{\max} M_n(\cl{T})$, we prove that the maximal tensor product is left projective. By symmetry, it is right projective and hence projective. 
\end{proof}

At last, we remark that this characterization of the maximal tensor product indeed coincides with the $(\min, \max)$-nuclearity result in \cite{HP, Ka}.  

\begin{cor}
Let $\delta_i$ be the dual basis of $y_i$ for $\cl{T}^d$. Then $u = \sum_{i=1}^m \delta_i \otimes y_i \in \cl{T}^d \otimes_{\max} \cl{T}$ is positive if and only if $\cl{T}$ is (min, max)-nuclear. 
\end{cor}

\begin{proof}
Let $\cl{S} = \cl{T}^d$ and note that $\hat{u}$ is the identity map on $\cl{T}$. Moreover, $u \in (\cl{T}^d \otimes_{\max} \cl{T} )^+$ if and only if $\hat{u}$ factors through $M_n$, which by \cite[Theorem 3.2]{HP}, if and only if $\cl{T}$ is (min, max)-nuclear. 
\end{proof}


\section{Acknowledgment}
The author would like to thank Vern I. Paulsen for his valuable advice and inspirations in writing this paper, and thank Prof. Vandana Rajpal for introducing the schur tensor product to the author.  

\providecommand{\bysame}{\leavevmode\hbox to3em{\hrulefill}\thinspace}
\providecommand{\MR}{\relax\ifhmode\unskip\space\fi MR }
\providecommand{\MRhref}[2]{%
  \href{http://www.ams.org/mathscinet-getitem?mr=#1}{#2}
}
\providecommand{\href}[2]{#2}

\end{document}